\documentclass{article}
\usepackage{latexsym}
\usepackage{amsmath,amsthm}
\usepackage{enumerate}
\usepackage{amssymb}
\usepackage{upgreek}

\usepackage[margin=1.1in,dvips]{geometry}

\newcommand{\vol}{\textnormal{vol}}

\def\f12{\frac 1 2}

\def\a{\alpha}
\def\b{\beta}
\def\ga{\gamma}

\def\ep{\epsilon}

\def\Si{\Sigma}
\def\om{\omega}
\def\Om{\Omega}

\def\pa{\partial}
\def\les{\lesssim}

\def\f12{\frac 1 2}

\newcommand{\lap}{\mbox{$\Delta \mkern-13mu /$\,}}

\newcommand{\nabb}{\mbox{$\nabla \mkern-13mu /$\,}}

\newtheorem{thm}{Theorem}

\newtheorem{prop}{Proposition}

\newtheorem{lem}{Lemma}

\newtheorem{remark}{Remark}
\begin{document}

\title{Global solutions of nonlinear wave equations with large energy}
\author{Shiwu Yang}
\date{}
\maketitle
\begin{abstract}
In this paper, we give a criterion on the Cauchy data for the semilinear wave equations satisfying the null condition in $\mathbb{R}^+\times\mathbb{R}^{3}$ such that the energy of the data can be arbitrarily large while the solution is still globally in time in the future.
\end{abstract}

\section{Introduction}
In this paper, we study the Cauchy problem to the semilinear wave equations
\begin{equation}
\label{largeeq}
\begin{cases}
 \Box\phi=(-\pa_t^2+\Delta)\phi=F(\phi, \pa\phi),\\
 \phi(0,x)=\phi_0(x), \pa_t\phi(0,x)=\phi_1(x)
\end{cases}
\end{equation}
in $\mathbb{R}^{+}\times \mathbb{R}^3$. The nonlinearity $F$ is assumed to
satisfy the null condition outside a large cylinder $\{(t, x)||x|\leq R\}$, that is,
\begin{equation}
 \label{nullcond}
F(\phi,\pa\phi)=A^{\a\b}\pa_\a\phi\pa_\b\phi+O(|\phi|^3+|\pa\phi|^3+|\phi|^{N}+|\pa\phi|^N), \quad |x|\geq R,
\end{equation}
where $A^{\a\b}$ are constants such that $A^{\a\b}\xi_\a\xi_\b=0$ whenever $\xi_0^2=\xi_1^2+\xi_2^2+\xi_3^2$ and $N$ is a given integer larger than 3. Inside the cylinder we assume $F$ is at least
quadratic in terms of $\phi$, $\pa\phi$. The data $(\phi_0, \phi_1)$ are assumed to be smooth but can be arbitrarily large in the energy space.

\bigskip

The long time behavior of solutions of nonlinear wave equations has drawn considerable attention in the past decades. When the data are small, the
classical result of Christodoulou \cite{ChDNull} and Klainerman \cite{klNull} shows that the solution of the nonlinear wave equation \eqref{largeeq}
is globally in time for all sufficiently small initial data. Generalizations and variants can be found in
\cite{soggestar},
\cite{klmulti},
\cite{smallglobLuk},
\cite{sogge-metcalfe-nakamura},
\cite{sogge-metcalfe},
\cite{sideris-multispeed},
\cite{soggemulti},
\cite{glsoChengbo}
and references therein. One of the most remarkable related results should be the global nonlinear stability
of Minkowski space first proved by Christodoulou-Klainerman \cite{kcg} and later an alternative proof contributed by Lindblad-Rodnianski \cite{SMigor}.

\bigskip

For the general large data case which is concerned in the current study, global existence results for the
 related wave map problems have been established with initial energy below that of any nontrivial harmonic maps,
see e.g.
\cite{wpKrieger},
\cite{wpTataru2},
\cite{wpTataru1},
\cite{wptao}.
Failure of this energy constriction may lead to finite time blow up, see e.g. \cite{wpbligor}, \cite{wpbligor2}, \cite{wpblshatah}, \cite{wpblstruwe}, \cite{wpbltataru}. For equation \eqref{largeeq}, the concrete example
\[
 \Box\phi=|\pa_t\phi|^2-|\nabla\phi|^2
\]
shows that the solution can blow up in finite time for general large data. For the details we refer to \cite{klgex}.

\bigskip

In a recent work \cite{largePin} of Wang-Yu, they constructed an open set of Cauchy data for the semilinear wave equations
satisfying the null condition such that the energy is arbitrarily large while the solution exists globally
in the future. The construction is indirect. They in fact impose the radiation data at the past null infinity and then solve the equation
to some time $t_0<0$ to obtain the Cauchy data. Their work relied on the \textsl{short pulse} method of Christodoulou in his monumental
work \cite{trapChristo} on the formation of trapped surface. Extensions and refinements of Christodoulou's result are contributed by, e.g.,
Klainerman-Rodnianski \cite{trapIgor}, Luk-Rodnianski \cite{trlukigor1}, \cite{trlukigor2}, Klainerman-Luk-Rodnianski \cite{trlukigor3},
Yu \cite{pindy}, \cite{pinenergy}, \cite{pinwp}.

\bigskip

However, the nonlinear terms considered in Wang-Yu's work are quite restrictive. In fact, only quadratic null forms are allowed and cubic or higher order nonlinearities are excluded for consideration due to the short pulse method. In this paper, we use the new approach developed in \cite{newapp}, \cite{yang1}, \cite{yang2}, \cite{yang3} to treat the nonlinear wave equations \eqref{largeeq} with large data. We are able to give a
criterion on the initial data such that the solution exists globally in the future while the energy can be arbitrarily large. In particular, our approach applies to equations with any higher order nonlinearities. Combined with the techniques developed in \cite{yang3}, our result here can even be extended to quasilinear wave equations.
Moreover, we no longer require the data to have compact support as in \cite{yang1}, \cite{yang2}, \cite{yang3}. This in particular implies that those results also hold for data satisfying conditions in this paper.

\bigskip

Before we state the main result, we define the necessary notations. We use the coordinate system $(t, x)=(x^0, x^1, x^2, x^3)$ of the Minkowski space.
We denote $\pa_0=\pa_t$, $\pa_i=\pa_{x^i}$, $\pa=(\pa_t, \pa_1, \pa_2, \pa_3)=(\pa_t, \nabla)$.
We may also use the standard polar coordinates $(t, r, \om)$. Let $\nabb$ denote the induced covariant derivative and $\lap$ the induced Laplacian on the spheres of constant $r$. We also define the null coordinates $u=\frac{t-r}{2}$, $v=\frac{t+r}{2}$ and denote the corresponding partial derivatives
\[
 \pa_u=\pa_t-\pa_r,\quad \pa_v=\pa_t+\pa_r, \quad \overline{\pa_v}=(\pa_v, \nabb)
\]
for $r>0$. The vector fields that will be used as commutators are
\[
 Z=\{\pa_t, \Om_{ij}=x_i\pa_j-x_j\pa_i\}.
\]
Let $\a$ be a positive constant. Without loss of generality, we may assume $\a<\frac{1}{4}$. Denote
\begin{align*}
 E_0(R)&= \sum\limits_{k\leq 4}\left.\int_{\{r\geq R\}\cap \mathbb{R}^3}r^{1+\a}|\overline{\pa_v}(rZ^k\phi)|^2drd\om\right|_{t=0}+
\left.\int_{\{r\leq R\}\cap \mathbb{R}^3}|\pa Z^k\phi|^2 dx\right|_{t=0},\\
 E_1(R)&=\sum\limits_{k\leq 4}\left.\int_{\{r\geq R\}\cap \mathbb{R}^3}|\pa_u(rZ^k\phi)|^2+|rZ^k\phi|^2drd\om\right|_{t=0}.
\end{align*}
These quantities can be uniquely determined by the initial data $(\phi_0,\phi_1)$ together with the equation \eqref{largeeq}. We have the
following main result:
\begin{thm}
 \label{mainthmlarge}
Consider the Cauchy problem for the semilinear wave equation \eqref{largeeq} satisfying the null condition \eqref{nullcond} with
some integer $N\geq 3$. For all $\a\in(0, 1)$, there exists a constant $R(\a)$, depending only on $\a$,
and a constant $\ep_0$, depending only on the highest order $N$ of the nonlinearity, such that if the initial data satisfy the estimate
\begin{equation}
\label{assumpla}
E_0(R)\leq R^{-2+\a},\quad E_1(R)\leq R^{\ep_0\a}
\end{equation}
for some $R\geq R(\a)$, then the solution $\phi$ exists globally in the future and obeys the estimates:
\begin{align*}
& |\overline{\pa_v}\phi|\leq C_\delta (1+r)^{-\frac{3}{2}+\delta},\quad \delta>0;\\
& |\pa_u\phi|\leq C_{\delta}(1+r)^{-1+\delta}(1+t-r+R)^{-\f12-\f12\a},\quad \delta>0, \quad t+R\geq r;\\
& |\pa_u\phi|\leq C(1+r)^{-1}R^{\f12\ep_0\a},\quad t+R<r,
\end{align*}
where the constant $C_\delta$ depends on $\delta$, $\a$ and the constant $C$ depends only on $\a$.
\end{thm}
\begin{remark}
Similar result holds for equations in higher dimensions without assuming the null condition.
\end{remark}
The Theorem implies that the energy of the initial data can be as large as $R^{\ep_0\a}$. Since $R$ can be any constant larger than a fixed constant $R(\a)$, the energy can be arbitrarily large. Moreover, the amplitude of the solution, at least in a small region, can have size $R^{\f12 \ep_0\a}$. In Wang-Yu's work, the construction of the Cauchy data is indirect and only the size of the energy has a lower bound. The amplitude or the $L^\infty$ estimates of the solution is unclear except the upper bound. From this point of view, the problem we consider here is a large data problem.

\bigskip

The existence of the initial data $(\phi_0, \phi_1)$ satisfying the conditions in the Theorem can be seen as follows: for any fixed $\a\in(0, 1)$ and any $R\geq R(\a)$, let $\phi_0$ be small in the ball with radius $R$ in $\mathbb{R}^3$. Here $R(\a)$ is a sufficiently large constant depending only on $\a$. Outside the ball, the energy of $\phi_0$ are allowed to be as large as $R^{\ep_0\a}$. Then for $\phi_1$, we require it to be small inside the ball with radius $R$. Outside the ball, it is close to $\pa_r\phi_0$. This will definitely give a large set of initial data $(\phi_0, \phi_1)$ satisfying the conditions in the Theorem.

\bigskip

We will use the new approach developed in \cite{newapp}, \cite{yang1}, \cite{yang2}, \cite{yang3} to prove the main Theorem. A key ingredient of this new approach is the $p$-weighted energy inequality originally introduced by Dafermos-Rodnianski \cite{newapp}. This inequality can be obtained by using the vector field $r^p\pa_v$ as multipliers in a neighborhood of the null infinity. It in particular implies that the $p$-weighted energy $E_0(R)$, see the definition before the main Theorem, keeps small if initially it is. This allows us to relax the size of the transversal derivative the solution, which is $E_1(R)$ in the theorem.

\bigskip

We will first construct the solution of the nonlinear wave equation outside the light cone, that is the region $r\geq R+t$, and show that the energy flux through the outgoing null hypersurface $r=t+R$ is small. And then we prove the solutions exists globally inside the light cone, for which we are not able to apply the results, e.g., in \cite{yang3} directly. In the previous results, the smallness needed in order to close the bootstrap argument for nonlinear problem is guaranteed by assuming the data to be sufficiently small. Hence
it is not necessary to keep track of the dependence of the radius $R$ of the constants in the argument. However, in this paper, the smallness comes from the radius $R$ and thus we need an argument with all the dependence on $R$.

\textbf{Acknowledgments}
The author would like to thank Igor Rodnianski for his consistent encouragement on this problem. He also thanks Pin Yu for helpful discussions.

\section{Preliminaries and energy identities }
We briefly recall the energy identity for wave equations, for details we refer to \cite{yang3}. Let $m$ be the Minkowski metric. We make a convention that the
Greek indices run from $0$ to $3$ while the Latin
indices run from $1$ to $3$. We raise and lower indices of any tensor relative to the metric $m$, e.g., $\pa^\ga=m^{\ga\mu}\pa_\mu$. Recall the energy-momentum tensor
\[
{\mathbb T}_{\mu\nu}[\phi]=\pa_\mu\phi\pa_\nu\phi-\frac12 m_{\mu\nu}\pa^{\gamma}\phi\pa_{\gamma}\phi.
\]
Given a vector field $X$, we define the currents
\[
J^X_\mu[\phi]= {\mathbb T}_{\mu\nu}[\phi]X^\nu, \qquad
K^X[\phi]= {\mathbb T}^{\mu\nu}[\phi]\pi^X_{\mu\nu},
\]
where $\pi^X_{\mu\nu}=\frac12 \mathcal{L}_Xg_{\mu\nu}$ is the deformation tensor of the vector field $X$. For any function $\chi$, we define the vector field $\tilde{J}^{X}[\phi]$
\begin{equation}
\label{mcurent} \tilde{J}^X[\phi]=\tilde{J}_{\mu}^X[\phi]\pa^\mu=\left(J_{\mu}^X[\phi] -
\f12\pa_{\mu}\chi \cdot\phi^2 + \f12 \chi\pa_{\mu}\phi^2\right)\pa^\mu.
\end{equation}
For any bounded region $\mathcal{D}$ in $\mathbb{R}^{3+1}$, using Stokes' formula, we have the energy identity
\begin{align}
\label{energyeq} \iint_{\mathcal{D}}\Box_g\phi
(\chi\phi+X(\phi))
 + K^X[\phi] +\chi\pa^\ga\phi\pa_\ga\phi-\f12 \Box_g\chi \cdot \phi^2d\vol=\int_{\pa \mathcal{D}}i_{\tilde{J}^X[\phi]}d\vol,
\end{align}
where $\pa\mathcal{D}$ denotes the boundary of the domain $\mathcal{D}$ and $i_Y d\vol$ denotes the contraction of the volume form $d\vol$
with the vector field $Y$ which gives the surface measure of the
boundary.

\section{The solution on the region $\{r\geq t+R\}$}
\label{section2}
In this section, we construct the solution of the nonlinear wave equation \eqref{largeeq} on the region $\{r\geq R+t\}$.
First we define some notations. For $R\leq r_1\leq r_2$, we use $S_{r_1, r_2}$ to denote the following outgoing null hypersurface emanating from the sphere with radius $r_1$
\[
 S_{r_1, r_2}:=\{u=-\frac{r_1}{2},\quad r_1\leq r\leq r_2\}.
\]
Similarly define $\bar C_{r_1, r_2}$ to be the following incoming null hypersurface emanating from the sphere with radius $r_2$
\[
 \bar C_{r_1, r_2}:=\{v=\frac{r_2}{2},\quad r_1\leq r\leq r_2\}.
\]
On the initial hypersurface $\mathbb{R}^3$, the annulus with radii $r_1$, $r_2$ is
\[
 B_{r_1, r_2}:=\{t=0,\quad r_1\leq r\leq r_2\}.
\]
We use $S_{r}$ to be short for $S_{r, \infty}$. Similarly we have $\bar C_{r}$ and $B_{r}$.

We use $\mathcal{D}_{r_1, r_2}$ to denote the region bounded by $S_{r_1, r_2}$, $B_{r_1, r_2}$, $\bar C_{r_1, r_2}$. Let $E[\phi](\Si)$ to be the energy flux for $\phi$ through
the hypersurface $\Si$ in the Minkowski space. In particular,
\[
 E[\phi](S_{r_1, r_2})=\int_{S_{r_1, r_2}}|\overline{\pa_v}\phi|^2r^2dvd\om,\quad E[\phi](\bar C_{r_1, r_2})=\int_{\bar C_{r_1, r_2}}|\overline{\pa_u}\phi|^2 r^2dud\om,
\]
where $\overline{\pa_u}=(\pa_u, \nabb)$. On the initial hypersurface
\[
 E[\phi](B_{r_1, r_2})=\int_{B_{r_1, r_2}}|\pa\phi|^2 dx.
\]
\subsection{Energy estimates}
In the energy identity \eqref{energyeq}, take the region $\mathcal{D}$ to be $\mathcal{D}_{r_1, r_2}$, the vector field $X=\pa_t$ and the function $\chi=0$. We
obtain the classical energy estimate
\begin{equation}
 \label{enesout}
2\iint_{D_{r_1, r_2}}\Box\phi\cdot \pa_t\phi
 d\vol+E[\phi](S_{r_1, r_2})+E[\phi](\bar C_{r_1, r_2})=E[\phi](B_{r_1, r_2}).
\end{equation}
We also need an integrated energy estimate adapted to the region $\mathcal{D}_{r_1, r_2}$. For some small positive constant $\ep$, depending only on $\a$, we construct
the vector field $X$ and choose the functions $f$, $\chi$ as follows
\[
X=f(r)\pa_r,\quad f=2\ep^{-1}-\frac{2\ep^{-1}}{(1+r)^{\ep}},\quad \chi=r^{-1}f.
\]
We then can derive from the energy identity \eqref{energyeq} that
\begin{equation}
\label{ILElg00}
I^{\ep}[\phi]_{r_1}^{r_2}\leq C_{\ep} (E[\phi](B_{r_1})+E[\phi](S_{r_1, r_2})+E[\phi](\bar C_{r_1, r_2})+D^{\ep}[\Box\phi]_{r_1}^{r_2}),
\end{equation}
where we denote
\[
 I^{\ep}[\phi]_{r_1}^{r_2}:=\iint_{\mathcal{D}_{r_1, r_2}}\frac{|\bar \pa\phi|^2}{(1+r)^{1+\ep}} dxdt,\quad D^{\ep}[F]_{r_1}^{r_2}:=
\iint_{\mathcal{D}_{r_1, r_2}}(1+r)^{1+\ep} |F|^2dxdt.
\]
Here $\bar \pa\phi=(\pa\phi, \frac{\phi}{1+r})$. The constant $C_\ep$ depends only on $\ep$ and is independent of $r_1$, $r_2$. For the derivation of the above estimate \eqref{ILElg00},
it is almost the same as Proposition 1 of \cite{yang3} or Proposition 2 of \cite{yang1}. The only point we have to point out here is that we use the fact that the
solution $\phi$ goes to zero as $r\rightarrow \infty$ on the initial
hypersurface. We thus can use a Hardy's inequality to control the integral of $\frac{|\phi|^2}{(1+r)^{2}}$. This is also the reason that we have $E[\phi](B_{r_1})$, which is
$E[\phi](B_{r_1, \infty})$ according to our notations, instead of $E[\phi](B_{r_1, r_2})$ on the right hand side of the above estimate \eqref{ILElg00}.

Combine the above two estimates \eqref{enesout}, \eqref{ILElg00}. We derive the following integrated energy estimates
\begin{prop}
 \label{propILEout}
We have
\begin{align}
\label{ILElarout}
E[\phi](S_{r_1, r_2})+E[\phi](\bar C_{r_1, r_2})+I^{\ep}[\phi]_{r_1}^{r_2}\leq C_{\ep} (E[\phi](B_{r_1})+D^\ep[\Box\phi]_{r_1}^{r_2})
\end{align}
for some constant $C_\ep$ depending only on $\ep$.
\end{prop}
\begin{proof}
 For the derivation of the integrated energy estimate \eqref{ILElg00}, we refer to \cite{yang1} or \cite{yang3}. Then from the energy identity
\eqref{enesout}, we can estimate
\begin{align*}
 E[\phi](S_{r_1, r_2})+E[\phi](\bar C_{r_1, r_2})\leq E[\phi](B_{r_1, r_2})+\f12 C_\ep^{-1} I^\ep[\phi]_{r_2}^{r_2}+2C_\ep D^\ep[\Box\phi]_{r_2}^{r_2},
\end{align*}
where $C_\ep$ is the constant in the integrated energy estimate \eqref{ILElg00}. Then the integrated energy estimate \eqref{ILElg00} can be
improved to be
\[
 I^{\ep}[\phi]_{r_1}^{r_2}\leq 4C_{\ep} (E[\phi](B_{r_1})+D^{\ep}[\Box\phi]_{r_1}^{r_2}).
\]
This together with the previous estimate proves the proposition. Here according to our notation $C_\ep$ is a constant depending only on the
small constant $\ep$.
\end{proof}

Next we consider the $p$-weighted energy inequality. In the energy identity \eqref{energyeq}, we take
\[
 X=f\pa_v,\quad \chi=r^{p-1},\quad f=r^p,\quad 0\leq p\leq 2.
\]
We can compute
\begin{align*}
 &\int_{B_{r_1, r_2}}i_{\tilde{J}^X[\phi]}d\vol=\f12\int_{B_{r_1, r_2}}f(|\pa_v\psi|^2+|\nabb\psi|^2)-\pa_r(f r\phi^2)+f' r\phi^2  drd\om,\\
&\int_{S_{r_1, r_2}}i_{\tilde{J}^X[\phi]}d\vol=\int_{S_{r_1, r_2}}f|\pa_v\psi|^2-\f12 \pa_v(fr\phi^2)dvd\om,\\
&\int_{\bar C_{r_1, r_2}}i_{\tilde{J}^X[\phi]}d\vol=-\int_{\bar C_{r_1, s_2}}f|\nabb\psi|^2+f'r \phi^2+\f12 \pa_u(fr\phi^2)dud\om,\\
&\iint_{\mathcal{D}_{r_1, r_2}}K^X[\phi] +\chi\pa^\ga\phi\pa_\ga\phi-\f12 \Box\chi \cdot \phi^2d\vol\\&=\iint_{\mathcal{D}_{r_1, r_2}}\f12
f'|\pa_v\psi|^2 +(\chi-\f12 f')|\nabb\psi|^2- \f12\pa_v(f'r\phi^2)drdt.
\end{align*}
Here $\psi=r\phi$. We can do integration by parts on $\mathcal{D}_{r_1,r_2}$ to estimate the integral of $\pa_v(f'r\phi^2)$. Alternatively, we
can modify the current vector field $\tilde{J}^X[\phi]$ defined in line \eqref{mcurent} to be
\[
 \hat J^X[\phi]=\tilde{J}^X[\phi]+\f12 f'r\phi^2 \pa_v.
\]
Notice that
\begin{align*}
 -\int_{B_{r_1, r_2}}\pa_r(fr\phi^2)drd\om- \int_{\bar C_{r_1, r_2}}\pa_u(fr\phi^2)dud\om+\int_{S_{r_1, r_2}}\pa_v(fr\phi^2)dvd\om=0.
\end{align*}
Then from the energy identity \eqref{energyeq} and the above calculations, we obtain
\begin{equation}
\label{pWEout}
\begin{split}
&\iint_{\mathcal{D}_{r_1, r_2}}r^{p-1}(p|\pa_v\psi|^2 +(2-p)|\nabb\psi|^2)drdtd\om+\int_{\bar C_{r_1,
r_2}}r^p|\nabb\psi|^2dud\om\\&+\int_{S_{r_1, r_2}}r^p|\pa_v\psi|^2dvd\om=\int_{B_{r_1, r_2}}r^p|\overline{\pa_v}\psi|^2dr
d\om-2\iint_{\mathcal{D}_{r_1, r_2}}r^{p-1}\Box \phi \pa_v\psi dxdt.
\end{split}
\end{equation}
\subsection{Bootstrap argument}
We assume initially
\[
 E_0(R)\leq R^{-\b},\quad E_1(R)\leq R^{\ep_1}.
\]
for some positive constant $\b$, which will be determined later. The definition of $E_0(R)$ can be found in the introduction before
the statement of the main Theorem. We impose the following bootstrap assumption on the nonlinearity $F(\pa\phi)$ in the equation \eqref{largeeq}
\begin{equation}
 \label{bootout}
\sum\limits_{k\leq 4}\iint_{\mathcal{D}_{R}}|Z^k F|^2 r^{2+\a}dxdt\leq 2 R^{-\b}.
\end{equation}
Then the $p$-weighted energy inequality \eqref{pWEout} obtained in the end of the previous subsection implies that
\begin{align*}
  &\iint_{\mathcal{D}_{r_1, r_2}}r^{p-1}(p|\pa_v Z^k\psi|^2 +(2-p)|\nabb Z^k\psi|^2 )drdtd\om+\int_{S_{r_1, r_2}}r^p|\pa_v Z^k\psi|^2dvd\om+\int_{\bar C_{r_1,
r_2}}r^p|\nabb Z^k\psi|^2dud\om\\
&\leq r_1^{p-1-\a}\int_{B_{r_1, r_2}}r^{1+\a}|\overline{\pa_v}\psi|^2drd\om+\iint_{\mathcal{D}_{r_1, r_2}}\frac{p}{2}r^{p-1}|\pa_v Z^k\psi|^2drdtd\om
+\iint_{\mathcal{D}_{r_1, r_2}}\frac{2}{p}r^{p+1}|Z^k\Box\phi|^2dxdt\\
&\leq r_1^{p-1-\a}E_0(R)+\frac{p}{2}\iint_{\mathcal{D}_{r_1, r_2}}r^{p-1}|\pa_v Z^k\psi|^2drdtd\om+
r_1^{p-1-\a}\iint_{\mathcal{D}_{r_1, r_2}}\frac{2}{p}r^{2+\a}|Z^k\Box\phi|^2dxdt.
\end{align*}
The second term in the last line can be absorbed. Then let $r_2$ goes to infinity, we can obtain the following $p$-weighted energy estimate
\begin{equation}
 \label{pwelarge}
\begin{split}
 &\iint_{\mathcal{D}_{r_1}}r^{p-1}|\overline{\pa_v} Z^k\psi|^2 drdtd\om+\int_{S_{r_1}}r^p|\pa_v Z^k\psi|^2dvd\om+\int_{\bar C_{r_1,
r_2}}r^p|\nabb Z^k\psi|^2dud\om\les R^{-\b}r_1^{p-1-\a}
\end{split}
\end{equation}
for all $k\leq 4$, $r_2\geq r_1\geq R$, $0<p\leq 1+\a$. Here and in the following we make a convention that $A\les B$ means $A\leq CB$ for some constant $C$ depending only on $\a$ and is independent of $R$, $r_1$.

\bigskip

Note that the assumption \eqref{assumpla} in particular implies that
\begin{equation}
\label{phiinbd}
 \int_{\om}r^2|Z^k \phi(0, r, \om)|^2d\om \les R^{\ep_1},\quad k\leq 4, \quad r\geq R.
\end{equation}
Using the $p$-weighted energy inequality \eqref{pwelarge} when $p=1+\a$, on $S_{r_1}$, we can estimate
\begin{align}
\notag
 \int_{\om}|r Z^k\phi|^2(t, r, \om)d\om&\leq \int_{\om} |r Z^k\phi(0, r_1, \om)|^2d\om+\int_{S_{r_1}}r^{1+\a}|\pa_v (r Z^k\phi)|^2dvd\om \cdot \a^{-1}r_1^{-\a}\\
\label{phi1bd}
&\les \int_{\om} |r Z^k\phi(0, r_1, \om)|^2d\om+ R^{-\b}r_1^{-\a},\quad k\leq 4.
\end{align}
In particular, we have
\[
 \int_{\om}|r Z^k\phi|^2(t, r, \om)d\om\les R^{\ep_1},\quad k\leq 4.
\]
We also need an inequality to estimating $\pa_u\psi$, $\psi=r\phi$. From the energy inequality \eqref{ILElarout}, we can show that
\begin{equation}
 \label{uenergy}
\begin{split}
\int_{\bar C_{R, r_2}}|\overline{\pa_u} Z^k\psi|^2dud\om\les E[\phi](\bar C_{R, r_2})+\int_{\om}r |Z^k\phi|^2(u_{r_2},v_{r_2}, \om)d\om\les  R^{\ep_1},\quad k\leq 4,\quad r_2\geq R,
\end{split}
\end{equation}
where $u_r=\frac{r-R}{2}$, $v_r=\frac{r+R}{2}$, $\overline{\pa_u}=(\pa_u, \nabb)$.

\bigskip

We now improve the bootstrap assumption \eqref{bootout}. The quadratic part of the nonlinearity $F$ is a null form $Q(\phi, \phi)$. Note that
\[
 Z^k Q(\phi, \phi)=\sum\limits_{k_1+k_2\leq k}Q(Z^{k_1}\phi, Z^{k_2}\phi).
\]
Here $Q$ denotes a general null form. They may stand for different null forms with different constants $A^{\mu\nu}$. For details we refer to e.g.
\cite{klNullc}. We denote
\[
\phi_1=Z^{k_1}\phi, \quad \phi_2=Z^{k_2}\phi,\quad \psi_1=r\phi_1, \quad \psi_2=r\phi_2.
\]
This $\phi_1$ is only a notation and should not be confused with the initial data $\phi_1$.
Note that
\[
 |r^2Z^kQ(\phi, \phi)|\les \sum\limits_{k_1+k_2\leq k}|\bar\pa\psi_1||\phi_2|+|\nabb\psi_1|(|\pa_t\psi_2|+|\pa_v\psi_2|)+
|\nabb\psi_1||\nabb\psi_2|+|\pa_u\psi_1||\pa_v\psi_2|.
\]
For the proof of the above inequality, see \cite{yang1}. Then by using Sobolev embedding, for $k\leq 4$, we have the estimate
\begin{equation}
\label{nullformest}
\begin{split}
 \int_{\om}|r^2 Z^k Q(\phi, \phi)|^2d\om \les  & \sum\limits_{k_1\leq 4, k_2\leq 4} \int_{\om}|\bar \pa\psi_1|^2d\om\cdot \int_{\om}|\phi_2|^2d\om+
\sum\limits_{k_1\leq 4, k_2\leq 2}\int_{\om}|\overline{\pa_u}\psi_1|^2d\om \cdot\int_{\om}|\pa_v\psi_2|^2d\om\\
&+\sum\limits_{k_1\leq 2, k_2\leq 4}
\int_{\om}|\pa_u\psi_1|^2d\om \cdot\int_{\om}|\pa_v\psi_2|^2d\om+\int_{\om}|\nabb Z^4\psi|^2|\pa_t\psi|^2d\om.
\end{split}
\end{equation}
We use estimate \eqref{phiinbd} to bound $\|\phi_2\|_{L^2(\mathbb{S}^2)}$, $k_2\leq 4$. For $\|\pa_v\psi_2\|_{L^2(\mathbb{S}^2)}$, $k_2\leq 3$, we can estimate
\begin{equation*}
\label{rapavpsi}
\begin{split}
 \sup\limits_{v=\frac{r}{2}, -\frac{r}{2}\leq u\leq -\frac{R}{2}}r^\a\|\pa_v\psi_2\|_{L^2(\mathbb{S}^2)}^2&\les r^\a\int_{\om}|\pa_v\psi_{2}|^2(0, r,\om)d\om+
\left|\int_{\bar C_{R, r}}\pa_u(r^\a|\pa_v\psi_2|^2)dud\om\right|\\
& \les r^\a\int_{\om}|\pa_v\psi_{2}|^2(0, r,\om)d\om+\int_{\bar C_{R, r}}r^{\a-1}|\pa_v\psi_2|^2+r^\a|\pa_u\pa_v\psi_2||\pa_v\psi_2|dud\om\\
&\les r^\a\int_{\om}|\pa_v\psi_{2}|^2(0, r,\om)d\om+\int_{\bar C_{R, r}}r^{\a-1}|\pa_v\psi_2|^2 dud\om\\
&\quad+\int_{\bar C_{R, r}}r^{\a+1}(|\nabb\Om\phi_2|^2+|rZ^{k_2}F|^2)dud\om,
\end{split}
\end{equation*}
where we have used the equation for $\phi_2$ in null coordinates $(u, v, \om)$. But the coordinate $(0, r, \om)$ appeared in the previous estimate is with respect to the polar coordinate
$(t, r, \om)$. Integrate the above estimate with respect to $r$ from $R$ to infinity. We obtain
\begin{equation}
\label{suppavpsi}
\begin{split}
 \int_{R}^\infty\sup\limits_{v=\frac{r}{2}, -\frac{r}{2}\leq u\leq -\frac{R}{2}}r^\a\|\pa_v\psi_2\|_{L^2(\mathbb{S}^2)}^2 dr &\les
\int_{B_R}r^\a|\pa_v\psi_2|^2drd\om +\iint_{\mathcal{D}_R}r^{\a-1}|\pa_v\psi_2|^2drdt d\om \\
&\quad+\iint_{\mathcal{D}_R}r^{\a-1}|\nabb \Om\psi_{2}|^2
+r^{\a+3}|Z^{k_2}F|^2 d td\om dr\\
&\les R^{-1-\b}.
\end{split}
\end{equation}
Here we have used the assumption on the initial data that $E_0(R)\leq R^{-\b}$. The bound for $F$ follows from the bootstrap assumption
\eqref{bootout}. The estimates for
$|\pa_v\psi_2|^2$, $|\nabb\Om \psi_2|^2$ are due to the $p$-weighted energy inequality \eqref{pwelarge} and the fact that $k_2\leq 3$.

Similarly, for $\|\pa_u\psi_1\|_{L^2(\mathbb{S}^2)}$, $k_1\leq 3$, we have
\begin{align*}
\int_{R}^{\infty}\sup\limits_{u=-\frac{r}{2}, v\geq\frac{r}{2}} r^{-1}\|\pa_u\psi_1\|_{L^2(\mathbb{S}^2)}^2dr
&\les R^{-1}\int_{B_{R}}|\pa_u\psi_1|^2drd\om+\iint_{\mathcal{D}_R}r^{-2}|\pa_u\psi_1|^2dtdrd\om\\
&+\iint_{\mathcal{D}_R}|\nabb \Om \phi_1|^2+|rZ^{k_1}F|^2drdtd\om\\
&\les R^{-1+\ep_1}+R^{-1+\ep}I^{\ep}[\phi_1]_{R}^{\infty}+R^{-2-\a-\b}\\
&\les R^{-1+\ep_1+\ep}.
\end{align*}
Here the estimate for the integrated energy estimate $I^\ep[\phi_1]_{R}^{\infty}$ follows from \eqref{ILElarout} in which the bounds for
$D^\ep[Z^{k_1}F]_{R}^{\infty}$ are guaranteed by the bootstrap assumption \eqref{bootout}.

\bigskip

On the right hand side of the null form estimate \eqref{nullformest}, we are left to estimate the special term
 $|\nabb \psi_1||\pa_t\psi|$, $\psi_1=Z^4\psi$. We can show that
\begin{align*}
 \iint_{\mathcal{D}_{R}}|\nabb\psi_1|^2|\pa_t\psi|^2 r^{\a}drdtd\om &\les\int_{R}^\infty\int_{\frac{r}{2}}^\infty (v+\frac{r}{2})^\a\int_{\om}|\nabb\psi_1|^2\cdot \int_{\om} |\psi_{2}|^2d\om drdv\\
&\les \int_{R}^\infty\int_{\frac{r}{2}}^\infty (v+\frac{r}{2})^\a\int_{\om}|\nabb\psi_1|^2\cdot (\int_{\om}|\psi_{2}|^2(r, \frac{r}{2},\om)d\om +R^{-\b}(v+\frac{r}{2})^{-\a}) drdv\\
&\les R^{-2\b-\a}+\int_{R}^{\infty}\int_{-\frac{r}{2}}^{-\frac{R}{2}}\int_{\om}(\frac{r}{2}-u)^\a|\nabb\psi_1|^2d\om \cdot \int_{\om} |\psi_{2}|^2(r, -\frac{r}{2},\om)d\om du dr.
\end{align*}
Here we have used estimate \eqref{phi1bd} and $k_2\leq 3$. We note that when $u$ is fixed, the $p$-weighted energy inequality \eqref{pwelarge} implies that
\begin{align*}
 \int_{-\frac{r}{2}}^{-\frac{R}{2}}\int_{\om}(\frac{r}{2}-u)^{1+\a}|\nabb\psi_1|^2d\om du\leq \int_{\bar C_{R, r}}r^{1+\a}|\nabb\psi_1|^2 dud\om \les R^{-\b}.
\end{align*}
Thus we can show that
\[
 \iint_{\mathcal{D}_{R}}|\nabb\psi_1|^2|\pa_t\psi|^2 r^{\a}drdtd\om \les R^{-2\b-\a}+\int_{R}^{\infty} R^{-\b-1}\int_{\om} |\psi_{2}|^2(0, r,\om)d\om dr\les R^{-2\b-\a}+R^{-1-\b+\ep_1}.
\]
Therefore from the null form estimate \eqref{nullformest}, we can derive that
\begin{align*}
 &\iint_{\mathcal{D}_{R}}r^{\a}|r^2 Z^kQ(\phi,\phi)|^2drdtd\om \\
&\les R^{\ep_1}\iint_{\mathbb{D}_R}r^{-2+\a}|\pa\psi_1|^2dtdrd\om+\int_{R}^\infty\sup\limits_{u}r^\a\|\pa_v\psi_2\|_{L^2(\mathbb{S}^2)}^2
\int_{\bar C_{R, r}}|\overline{\pa_u}\psi_1|^2dud\om dr\\
&\quad +\int_{R}^{\infty}\sup\limits_{v}r^{-1}\|\pa_u\psi_1\|_{L^2(\mathbb{S}^2)}^2 \int_{S_{r_1}}r^{1+\a}|\pa_v\psi_2|^2dvd\om dr_1+R^{-2\b-\a}+R^{-1-\b+\ep_1}\\
&\les R^{2\ep_1-1+\a+\ep}+R^{\ep_1-1-\b}+R^{-1+\ep_1+\ep-\b}+R^{-2\b-\a}+R^{-1-\b+\ep_1}.
\end{align*}
For cubic or higher order nonlinearities, we first conclude from estimate \eqref{suppavpsi} that
\[
 \int_{0}^{r_1-R}\int_{\om}r_1^{\a}(|\pa_vZ^k\psi|^2+|\pa_v\pa_t Z^k\psi|^2)d\om dt\les R^{-1-\b},\quad k\leq 2.
\]
In particular, we have
\[
 \int_{\om}|\pa_v Z^k\psi|^2d\om \les R^{-1-\b}r_1^{-\a}, \quad k\leq 2.
\]
Since we have shown that
\[
 \int_{\om}|Z^k\psi|^2d\om \les R^{\ep_1},\quad k\leq 4,
\]
we then have
\[
 \int_{\om}|\pa Z^k\psi|^2d\om\les R^{\ep_1},\quad k\leq 2.
\]
Thus for cubic or higher order nonlinearities, we can bound
\begin{align*}
\iint_{\mathcal{D}_{R}}|Z^k (F-Q)|^2 r^{2+\a}dxdt\les  \sum\limits_{k\leq 4}\iint_{\mathcal{D}_R}|\pa Z^k\phi|^2 r^{-4+2+\a}R^{2(N-2)\ep_1} dxdt
\les R^{(2N-3)\ep_1+\a+\ep-1}.
\end{align*}
Here we recall that $N$ is the order of the highest order nonlinearity.
To summarize, we have shown that
\[
 \iint_{\mathcal{D}_{R}}|Z^k F|^2 r^{2+\a}dxdt\les R^{(2N-3)\ep_1+\a+\ep-1}+R^{-1+\ep_1+\ep-\b}+R^{-2\b-\a}.
\]
If we take
\begin{equation}
\label{defb}
\b=1-2\a,\quad \ep=\frac{\a}{20},\quad \ep_1=\frac{\a}{2N},
\end{equation}
we then have
\[
 \iint_{\mathcal{D}_{R}}|Z^k F|^2 r^{2+\a}dxdt\les R^{-\b-\frac{1}{5}\a}.
\]
According to our notations, the implicit constant in the
above estimate depends only on $\a$. Hence let the constant $R$ be sufficiently large, depending only on $\a$,
we then can improve the bootstrap assumption \eqref{bootout}. Once we have improved the bootstrap assumption \eqref{bootout},
the proof for the existence of a unique solution of the equation \eqref{largeeq} on the region $\{r\geq R+t\}$ is standard,
see the end of \cite{yang1}.
\begin{remark}
 In particular, the small constant $\ep_0$ in the main Theorem can be $\ep_0=\frac{1}{2N}$.
\end{remark}

\section{The solution on $\{r\leq R+t\}$}
We have constructed the solution of the equation \eqref{largeeq} outside the light cone $\{r\geq R+t\}$. In this section, we will
prove that the solution also exists globally in the future inside the light cone which is the region
$\{r\leq R+t\}$.
We use the foliation
\[
 S_\tau:=\{u=u_\tau=\frac{\tau-R}{2},\quad \frac{\tau+R}{2}=v_\tau\leq v\},\quad \Si_\tau:=\{t=\tau, \quad r\leq R\}\cup S_\tau.
\]
The energy flux through $\Si_\tau$ for the scalar field $\phi$ is $E[\phi](\tau)$. For $\tau_2\geq \tau_1$, we define
\[
 I^\ep[\phi]_{\tau_1}^{\tau_2}:=\int_{\tau_1}^{\tau_2}\int_{\Si_\tau}\frac{|\bar\pa\phi|^2}{(1+r)^{1+\ep}}dxd\tau,\quad D^\ep[F]_{\tau_1}^{\tau_2}:=
\int_{\tau_1}^{\tau_2}\int_{\Si_\tau}(1+r)^{1+\ep}|F|^2dxd\tau.
\]
We have the integrated energy estimate and the energy estimate
\begin{equation}
\label{ILElarge}
 E[\phi](\tau_2)+I^{\ep}[\phi]_{\tau_1}^{\tau_2}+\int_{\tau_1}^{\tau_2}\int_{S_\tau}\frac{|\nabb\phi|^2}{1+r}dxd\tau\les E[\phi](\tau_1)+D^{\ep}[F]_{\tau_1}^{\tau_2},
\end{equation}
see Proposition 1 of \cite{yang3} or Proposition 2 of \cite{yang1}. As before, the implicit constant here depends only on $\ep$.
\subsection{The $p$-weighted energy inequality}
As we have discussed in the introduction, the smallness needed to close the bootstrap argument for nonlinear problem in this paper comes from the radius $R$ while in the previous work e.g. \cite{yang1} the smallness comes from the data. In particular, the previous argument can not be applied directly to the settings in this paper. Instead we need an argument with all the dependence of the constants on the radius $R$. To be more precise, we first consider one of the key ingredients the $p$-weighted energy inequality. We recall the $p$-weighted energy identity originally introduced by Dafermos-Rodnianski in \cite{newapp}
\begin{align*}
&\int_{S_{\tau_2}^v} r^p (\pa_v\psi)^2 dvd\om +\int_{\tau_1}^{\tau_2}\int_{S_\tau^v}2r^{p+1}F\cdot\pa_v\psi dvd\tau d\om\\
& +\int_{\tau_1}^{\tau_2}\int_{S_\tau^v} r^{p-1}  \left (p(\pa_v\psi)^2 +
(2-p) |\nabb\psi|^2\right)dvd\tau d\om +\int_{\bar C(\tau_1, \tau_2,v)} r^p |\nabb\psi|^2 du d\om\\
=& \int_{S_{\tau_1}^v}r^p (\pa_v\psi)^2 dvd\om +\int_{\tau_1}^{\tau_2} r^p \left (|\nabb\psi|^2- (\pa_v\psi)^2\right)d\om d\tau |_{r=R},
\end{align*}
where $\psi=r\phi$, $F=\Box\phi$. Note that the boundary term on $\{r=R\}$ is proportional to $R^p$. Hence we can simply take $p=0$ to estimate it. First for any $\tau$, we have
\[
\int_{S_\tau^v}(\pa_v\psi)^2dvd\om\leq 5 E[\phi](\tau).
\]
For the proof of this inequality, see e.g. Corollary 1 in \cite{yang1}.
For the inhomogeneous term $F\pa_v\psi$ when $p=0$, we can estimate it as follows:
\begin{align*}
|\int_{\tau_1}^{\tau_2}\int_{S_\tau^v}rF\cdot\pa_v\psi dvd\tau d\om |&\les D^\ep[F]_{\tau_1}^{\tau_2}+E[\phi](\tau_1).
\end{align*}
Therefore for general $p$, we have the estimate for the boundary term
  \begin{align*}
\left|\int_{\tau_1}^{\tau_2} r^p \left (|\nabb\psi|^2- (\pa_v\psi)^2\right)d\om d\tau |_{r=R}\right|\les R^p
(D^\ep[F]_{\tau_1}^{\tau_2}+E[\phi](\tau_1)).
\end{align*}
Since the boundary term on the incoming null hypersurface $\bar C(\tau_1, \tau_2, v)$ has a good sign, to obtain a useful
estimate from the $p$-weighted energy identity, it suffices to
estimate the integral of the inhomogeneous term $r^{p+1}F\pa_v\psi$ in the above $p$-weighted energy identity. On $S_\tau$, we
control it as follows
$$2r^{p+1}|F\pa_v\psi| \leq r^p|\pa_v\psi|^2\tau_+^{-1-\ep}+r^{p+2}|F|^2\tau_+^{1+\ep},\quad \tau_+=1+\tau.
$$
The integral of the first term $r^p|\pa_v\psi|^2\tau_+^{-1-\ep}$ will be bounded by using Gronwall's inequality. Thus we derive
 \begin{align} \notag
&\int_{S_{\tau_2}}r^p(\pa_v\psi)^2dv d\om+\int_{\tau_1}^{\tau_2}\int_{S_\tau}r^{p-1}(p|\pa_v\psi|^2+(2-p)|\nabb\psi|^2)dvd\om d\tau\\
\label{pWEineq} &\les R^p(E[\phi](\tau_1)+ D^\ep[F]_{\tau_1}^{\tau_2})+\int_{S_{\tau_1}}r^p|\pa_v\psi|^2dv d\om+ \int_{\tau_1}^{\tau_2}\tau_+^{\ep}D_+^{p-1}[F]_{\tau}^{\tau_2}d\tau
+(\tau_1)_+^{1+\ep}D_+^{p-1}[F]_{\tau_1}^{\tau_2},
\end{align}
where
\[
 D_+^{\a}[F]_{\tau_1}^{\tau_2}:=\int_{\tau_1}^{\tau_2}\int_{S_\tau}(1+r)^{1+\a}|F|^2dxd\tau.
\]

\subsection{The data}
To study the equation on the region $\{r\leq t+R\}$, we need the initial data on the outgoing null hypersurface $S_0$, that is
$\{v\geq \frac{R}{2}, u=-\frac{R}{2}\}$. The data on the ball with radius $R$ can be arbitrarily
small according to our assumptions. It suffices to understand the solution on the outgoing null hypersurface $S_0$
(or using the notation in Section \ref{section2} $S_{R, \infty}$). Recall that we already constructed the solution on the region $\{r\geq t+R\}$
in the previous section. From the $p$-weighted energy
inequality \eqref{pwelarge} we have
\begin{equation}
\label{data0}
 \int_{S_0}r^{1+\a}|\pa_v Z^k\psi|^2dvd\om \les R^{-\b}, \quad \b=1-2\a,\quad k\leq 4.
\end{equation}
Here note that we have fixed $\b$ in line \eqref{defb}.
For the energy flux, we can assume
\begin{equation}
\label{data1}
 \sum\limits_{k\leq 4}\int_{S_0}|\overline{\pa_v} Z^k\phi|^2r^2dvd\om \les R^{-\b-\a-1}.
\end{equation}
This is consistent with the previous inequality as $|\pa_v Z^k\psi|^2$ is the main part of $|\overline{\pa_v} Z^k\phi|^2r^2$. A rigorous way to see this is to use the $p$-weighted energy
inequality \eqref{pwelarge}. We have
\[
 \int_{R}^\infty \int_{S_{r, \infty}}r^\a|\overline{\pa_v}Z^k\psi|^2dvd\om dr\les R^{-\b}.
\]
Since the data inside the ball with radius $R$ is small, we also can show (simply replacing $R$ in Section \ref{section2} with $\f12 R$) that
\[
 \int_{\f12 R}^R \int_{S_{r, \infty}}r^\a|\overline{\pa_v}Z^k\psi|^2dr\les R^{-\b}.
\]
In particular, we can choose a slice such that
\[
 E[Z^k\phi](S_{r_0, \infty})\leq \int_{S_{r_0,\infty}}|\overline{\pa_v}Z^k\psi|^2dvd\om +r_0\int_{\om}|\phi|^2(0, r_0, \om)d\om  \les R^{-\b-\a-1}=R^{-2+\a}
\]
for some $r_0\in (\f12 R, R)$. Here we note that the initial data on the ball with radius $R$ can be arbitrarily small.

In particular the data on $S_0$ satisfy the above two estimates \eqref{data0} and \eqref{data1}. Here recall that the implicit constant depends only
on $\a$.

\subsection{Bootstrap argument}
We now use the above boundary conditions to establish the decay of the energy flux. We impose the following bootstrap assumptions on the nonlinearity $F$ for all $k\leq 4$
\begin{equation}
\label{bsFlarge}
D^\ep[Z^k F]_{\tau_1}^{\tau_2}\leq 2\min\{R^{-\b}(\tau_1)_+^{-1-\a}, R^{-2+\a}, R^{-1-\b-\ep}(\tau_1)_+^{-\a}\},\quad D^\a_{+}[Z^kF]_{\tau_1}^{\tau_2}\leq 2\tau_+^{-1-\a}R^{-\b}.
\end{equation}
We show the decay of $E[Z^k\phi](\tau)$. Let $p=1+\a$ in the $p$-weighted energy inequality \eqref{pWEineq}. We have
\[
 \int_{S_{\tau_2}}r^{1+\a}|\pa_v\psi|^2dvd\om+\int_{\tau_1}^{\tau_2}\int_{S_\tau}r^\a|\pa_v\psi|^2dvd\om d\tau \les R^{1+\a-2+\a}+R^{-\b}=R^{-\b}.
\]
Hence we can choose a dyadic sequence $\{\tau_n\}$ such that
\[
 \int_{S_{\tau_n}}r^{\a}|\pa_v\psi|^2dvd\om \les (\tau_{n})_+^{-1}R^{-\b}.
\]
Interpolation leads to
\[
 \int_{S_{\tau_n}}r|\pa_v\psi|^2dvd\om \les R^{-\b}(\tau_n)_+^{-\a}.
\]
Then take $p=1$ in the $p$-weighted energy inequality \eqref{pWEineq}. We derive
\begin{align*}
 \int_{\tau_n}^{\tau'}E[\phi](\tau)d\tau &\les R^{-\b}(\tau_n)_+^{-\a}+R E[\phi](\tau_n)+R^{1+\ep}(E[\phi](\tau_n)+(\tau_n)_+^{-\a}R^{-1-\b-\ep})\\
&\les R^{-\b}(\tau_n)_+^{-\a}+R^{1+\ep}E[\phi](\tau_n),\quad \tau'\geq \tau_n.
\end{align*}
In the energy estimate \eqref{ILElarge}, set $\tau_1=0$. We have
\[
 E[\phi](\tau)\les R^{-2+\a}.
\]
For $\tau'\geq \tau$, we have
\[
 E[\phi](\tau')\les E[\phi](\tau)+R^{-\b}\tau_+^{-1-\a}.
\]
We thus can conclude that
\[
 (\tau'-\tau_n)E[\phi](\tau')\les R^{-\b}(\tau_n)_+^{-\a}+R^{1+\ep}E[\phi](\tau_n).
\]
In particular, we have
\begin{equation*}
 E[\phi](\tau)\les \tau_+^{-1}(R^{-\b}+R^{1+\ep}R^{-2+\a})\les \tau_+^{-1}R^{-\b}.
\end{equation*}
This then implies that
\[
 E[\phi](\tau_{n+1})\les R^{-\b}(\tau_n)_+^{-1-\a}+R^{1+\ep -\b}(\tau_{n})_+^{-2}.
\]
As $\tau_n$ is dyadic, we then infer that
\[
 E[\phi](\tau)\les R^{-\b}\tau_+^{-1-\a}+R^{1+\ep-\b}\tau_+^{-2}.
\]
Summarizing, we have the following energy decay estimate
\begin{prop}
\label{prop2}
 For any $k\leq 4$, we have
\[
 I^\ep[Z^k \phi]_{\tau_1}^{\tau_2}+ D^\ep[Z^k F]_{\tau_1}^{\tau_2}+E[Z^k\phi](\tau)\les A(\tau),
\]
where
\[
 A(\tau):=\min\{R^{-\b}\tau_+^{-1-\a}+R^{1+\ep-\b}\tau_+^{-2},\quad R^{-2+\a},\quad R^{-\b}\tau_+^{-1} \}.
\]
In particular, we have
\begin{equation*}
 E[Z^k\phi](\tau)\les\min\{ R^{-\ga}\tau_+^{-1-\a},R^{-2+\a}\},\quad \ga=\b-(1+\ep)\a.
\end{equation*}
\end{prop}
\begin{proof}
The estimate for the energy flux $E[\phi](\tau)$ follows from
the above argument. The estimate for the integrated energy $I^\ep[Z^k\phi]_{\tau_1}^{\tau_2}$ follows from \eqref{ILElarge} and the bound
for the inhomogeneous term $F$ is a restatement of the bootstrap assumption \eqref{bsFlarge}.

\end{proof}
The following lemma will be used to show the $C^1$ estimate of the solution.
\begin{lem}
\label{uvwestlar}
 \[
  \int_{\tau_1}^{\tau_2}\int_{\Si_\tau\cap \{r\geq 1\}}r^{1-\ep} |\pa_u\pa_v Z^k\phi|^2 dxd\tau\les A(\tau_1),\quad \forall k\leq 3.
 \]
\end{lem}
\begin{proof}
 Using the equation for $Z^k\phi$ (commutation of the equation \eqref{largeeq} with $Z^k$) we have
\begin{align*}
 \int_{\tau_1}^{\tau_2}\int_{\Si_\tau\cap \{r\geq 1\}}r^{1-\ep} |\pa_u\pa_v Z^k\phi|^2 dxd\tau &\les \int_{\tau_1}^{\tau_2}\int_{\Si_\tau\cap \{r\geq 1\}}r^{1-\ep} (r^{-1}|\pa Z^k\phi|
+|\lap Z^k\phi|+|Z^k F|)^2 dxd\tau\\
& \les I^{\ep}[Z^k\phi]_{\tau_1}^{\tau_2}+I^\ep[\Om Z^k\phi]_{\tau_1}^{\tau_2}+D^\ep[ Z^kF ]_{\tau_1}^{\tau_2}\les A(\tau_1)
\end{align*}
for all $k\leq 3$.
\end{proof}
Next, we improve the bootstrap assumption \eqref{bsFlarge}. We mainly consider the quadratic nonlinearity $Q(\phi, \phi)$, which satisfies the null condition.
We first estimate $D^\a_+[F]_{\tau_1}^{\tau_2}$. On $S_{\tau}$, we can estimate
\begin{align*}
 \int_{\om}|r Z^k\phi|^2(\tau, r, \om)d\om&\leq \int_{\om} |r Z^k\phi(\tau, R, \om)|^2d\om+\int_{S_{\tau}}r^{1+\a}|\pa_v (r Z^k\phi)|^2dvd\om \cdot \a^{-1}R^{-\a}\\
&\les \int_{\om} |r Z^k\phi(\tau, R, \om)|^2d\om+ R^{-\b}R^{-\a}\les R^{-1+\a}.
\end{align*}
Let $\bar C_{\tau_1, \tau_2, v_1}$ be the incoming null hypersurface between $\Si_{\tau_1}$ and $\Si_{\tau_2}$, defined as follows:
\[
 \bar C_{\tau_1, \tau_2, v_1}:=\{v=v_1, \quad u_{\tau_1}\leq u\leq u_{\tau_2}\}.
\]
The energy estimate on the region $\{v\geq v_1, \quad u_{\tau_1}\leq u\leq u_{\tau_2}\}$ then implies that
\[
\int_{C_{\tau_1,\tau_2, v_1}}|\overline{\pa_u}Z^k\psi|^2dud\om \les A(\tau_1),\quad k\leq 4.
\]
For the detailed proof of this estimate, we refer to e.g. Lemma 8 in \cite{yang1} or Lemma 11 in \cite{yang2}.
Then from estimate \eqref{nullformest}, we can show that
\begin{align*}
 D_+^\a[Z^k Q]_{\tau_1}^{\tau_2}&\les  R^{-1+\a}\int_{\tau_1}^{\tau_2}\int_{S_\tau}|\bar\pa\psi_1|^2 r^{-3+\a} drdtd\om+\sum\limits_{k_1\leq 2}
R^{-\b}\int_{\tau_1}^{\tau_2}\sup\limits_{v}r^{-2}\int_{\om}|\pa_u\psi_1|^2d\om d\tau\\
&\quad+A(\tau_1)\sum\limits_{k_2\leq 2}
\int_{v_{\tau_1}}^{\infty}\sup\limits_{u} r^{\a-1}\int_{\om}|\pa_v\psi_{2}|^2d\om dv+\int_{\tau_1}^{\tau_2}\int_{S_\tau}|\nabb Z^4\psi|^2 r^{\a} E[Z^3\phi](\tau)drd\om d\tau.
\end{align*}
Here we still use the notation that $\phi_1=Z^{k_1}\phi$, $\phi_2=Z^{k_2}\phi$, $\psi_1=r\phi_1$.
Now on $S_\tau$, $\tau_1\leq \tau\leq \tau_2$, we can estimate
\begin{align*}
 r^{-\a}\int_{\om}(\pa_u\psi_{1})^2d\om &\les\left.r^{-\a}\int_{\om}(\pa_u\psi_{1})^2d\om\right|_{v=v_{\tau_2}}+\int_{S_\tau}r^{-1-\a}|\pa_u\psi_{1}|^2dvd\om\\
&\quad\quad + \int_{S_\tau}r^{-1-\a}(\pa_u\psi_{1})^2dvd\om + \int_{S_\tau}r^{1-\a}(\pa_v\pa_u\psi_{1})^2dvd\om\\
&\les\left.\int_{\om}(\pa_u\psi_{1})^2d\om\right|_{v=v_{\tau_2}}+ \int_{S_\tau}r^{-1-\ep}(|\pa\psi_{1}|^2+|\pa\Om \psi_1|^2)dvd\om+\int_{S_\tau}r^{3-\a}|Z^{k_1}F|^2dvd\om.
\end{align*}
Similarly, on $\bar C_{\tau_1, \tau_2, v}$, we have
\begin{align*}
r^{\a}\int_{\om}(\pa_v\psi_{2})^2d\om &\les\left.r^{\a}\int_{\om}(\pa_v\psi_{2})^2d\om\right|_{u=u_{\tau_1}} + \int_{\bar C_{\tau_1,\tau_2, v}}r^{\a}(\pa_v\psi_{2})^2dud\om \\
&\quad\quad+ \int_{\bar C_{\tau_1, \tau_2, v}}r^{\a}(\pa_u\pa_v\psi_{2})^2dud\om + \int_{\bar C_{\tau_1, \tau_2, v}}r^{\a-1}(\pa_v\psi_{2})^2dud\om\\
&\les \left.r^{\a}\int_{\om}(\pa_v\psi_{2})^2d\om\right|_{u=u_{\tau_1}} + \int_{\bar C_{\tau_1, \tau_2, v}}r^{\a}(\pa_v\psi_{2})^2+r^{\a}(\lap\psi_{2})^2+
r^{\a+2}|Z^{k_2}F|^2dud\om.
\end{align*}
Therefore we can show that
\begin{align*}
D_+^\a[Z^k Q]_{\tau_1}^{\tau_2}\les R^{-3+2\a+\ep}A(\tau_1)+ R^{-\b-2+\a}A(\tau_1)+A(\tau_1)R^{-1-\b}+A(\tau_1)R^{-\b}\les A(\tau_1)R^{-\b}.
\end{align*}
The estimate for cubic or higher order nonlinearities is better and we can conclude that
\begin{equation}
\label{Douta}
 D_+^\a[Z^k F]_{\tau_1}^{\tau_2}\les A(\tau_1)R^{-\b},\quad k\leq 4.
\end{equation}
Next we estimate the integral inside the cylinder with radius $R$. We have
\begin{align*}
 \int_{\tau_1}^{\tau_2}\int_{r\leq R}(1+r)^{1+\ep}|Z^k F|^2 dxd\tau&\les \int_{\tau_1}^{\tau_2}\int_{r\leq R}(1+r)^{1+\ep}|\pa\phi_1|^2|\pa\phi_2|^2 dxd\tau\\
&\les\int_{\tau_1}^{\tau_2}\int_{r\leq 1}|\pa\phi_1|^2|\pa\phi_2|^2 dxd\tau+\int_{\tau_1}^{\tau_2}\int_{1\leq r\leq R}r^{1+\ep}|\pa\phi_1|^2|\pa\phi_2|^2 dxd\tau.
\end{align*}
Here we omitted the summation sigh for simplicity and the right hand side should be interpreted as the sum for all
$k_1+k_2\leq k\leq 4$. The integral on the cylinder with radius $1$ can be estimated by using elliptic estimates, which relies on the
 commutator $\pa_t$. To estimate the second part, we claim that
\begin{equation}
\label{claim}
 \int_{\om}r|\pa (Z^k\phi)|^2d\om\les  A(\tau),\quad k\leq 2,\quad 1\leq r\leq R.
\end{equation}
In fact from Lemma \ref{uvwestlar}, we have
\begin{align*}
 &\int_{1\leq r\leq R}r^{1-\ep}|\pa_u\pa_v Z^k\phi|^2 dx\les A(\tau),\quad k\leq 2,\\
&\int_{1\leq r\leq R}|\pa_u\pa_t Z^k\phi|^2 dx\les E[\pa_t Z^k\phi](\tau)\les A(\tau),\quad k\leq 3.
\end{align*}
This implies that
\[
 \int_{1\leq r\leq R}|\pa_u\pa_r Z^k\phi|^2 dx\les A(\tau),\quad k\leq 2.
\]
In particular, we can show that
\[
 r\int_{\om}|\pa_u Z^k\phi|^2d\om \leq A(\tau),\quad 1\leq r\leq R,\quad k\leq 2.
\]
This leads to the above claim \eqref{claim}. Hence, we can show that
\begin{align*}
 \int_{\tau_1}^{\tau_2}\int_{1\leq r\leq R}r^{1+\ep}|\pa\phi_1|^2|\pa\phi_2|^2 dxd\tau\les R^{\ep}\int_{\tau_1}^{\tau_2}A(\tau)^2d\tau\les A(\tau_1)R^{\ep-\b}.
\end{align*}
Inside the cylinder with radius $1$, by using elliptic theory, we can show that
\[
 |\pa Z^k\phi|^2\les A(\tau),\quad k\leq 2.
\]
For the details, we refer to e.g. the end of the second last section of \cite{yang3}. Therefore, we can estimate
\[
 \int_{\tau_1}^{\tau_2}\int_{r\leq1}|\pa\phi_1|^2|\pa\phi_2|^2 dxd\tau\les \int_{\tau_1}^{\tau_2}A(\tau)^2d\tau\les A(\tau_1)R^{-\b+\ep}.
\]
Combined with the estimate \eqref{Douta}, we then have shown that
\begin{align*}
 D^{\ep}[Z^k F]_{\tau_1}^{\tau_2}\leq D^\a_+[Z^k F]_{\tau_1}^{\tau_2}+ \int_{\tau_1}^{\tau_2}\int_{r\leq R}(1+r)^{1+\ep}|Z^k F|^2 dxd\tau
\les A(\tau_1)R^{\ep-\b},\quad \forall k\leq 4.
\end{align*}
Simply considering the total decay in $R$ and $(\tau_1)_+$, we see from the definition of $A(\tau)$ in Proposition
\ref{prop2} that
\[
 A(\tau)\leq 2 \tau_+^{-1-\a}R^{\ep+\a-\b}.
\]
Therefore  we have the estimate for $ D^\a_+[Z^k F]_{\tau_1}^{\tau_2}$
\[
  D^\a_+[Z^k F]_{\tau_1}^{\tau_2}\les (\tau_1)_+^{-1-\a}R^{-\b}R^{\ep+\a-\b},\quad \forall k\leq 4.
\]
For $D^{\ep}[Z^k F]_{\tau_1}^{\tau_2}$, when $(\tau_1)_+\leq R$, we have
\[
A(\tau_1)\leq R^{-2+\a}\leq R^\ep\min\{R^{-\b}(\tau_1)_+^{-1-\a}, R^{-2+\a}, R^{-1-\b-\ep}(\tau_1)_+^{-\a}\}.
\]
Here recall that $\b=1-2\a$, $\ep=\frac{\a}{20}$. When $(\tau_1)_+\geq R$, we can show that
\begin{align*}
 A(\tau_1)&\leq R^{-\b}(\tau_1)_+^{-1-\a}+R^{1+\ep-\b}(\tau_1)_+^{-2}\leq 2R^{1+\ep-\b}(\tau_1)_+^{-2}\\
&\leq 2R^{2\ep+\a} \min\{R^{-\b}(\tau_1)_+^{-1-\a}, R^{-2+\a}, R^{-1-\b-\ep}(\tau_1)_+^{-\a}\}
\end{align*}
In any case, we have
\[
 A(\tau_1)\leq 2R^{2\ep+\a} \min\{R^{-\b}(\tau_1)_+^{-1-\a}, R^{-2+\a}, R^{-1-\b-\ep}(\tau_1)_+^{-\a}\}
\]
Therefore we have
\[
 D^{\ep}[Z^k F]_{\tau_1}^{\tau_2}\les A(\tau_1)R^{\ep-\b}\les R^{3\ep+\a-\b}\min\{R^{-\b}(\tau_1)_+^{-1-\a}, R^{-2+\a},
R^{-1-\b-\ep}(\tau_1)_+^{-\a}\}
\]
for all $k\leq 4$. Recall that $\ep=\frac{\a}{20}$, $\b=1-2\a$ and $\a<\frac{1}{4}$. We conclude that for sufficiently large $R$,
depending only on $\a$, we can improve the bootstrap assumption \eqref{bsFlarge}. Then the construction
of the solution on the region $\{r\leq t+R\}$ will be the same as that in e.g. \cite{yang1} (the last section). Hence we can conclude
our main Theorem.

\bibliography{shiwu}{}
\bibliographystyle{plain}

\bigskip

DPMMS, Centre for Mathematical Sciences, University of Cambridge,
Wilberforce Road, Cambridge, UK CB3 0WA

\textsl{Email address}: S.Yang@damtp.cam.ac.uk
 \end{document}